\documentclass[12pt]{amsart}
\usepackage{latexsym, amsmath, amssymb,amsthm,amsopn,amsfonts}
\usepackage{version}
\usepackage{epsfig,graphics,color,graphicx,graphpap}
\usepackage{amssymb}
\usepackage{todonotes}
\usepackage{listings}
\usepackage{mathrsfs}
\usepackage[citecolor=blue]{hyperref}

\usepackage[framemethod=tikz]{mdframed}
\newcommand{\redsout}{\bgroup\markoverwith{\textcolor{red}{\rule[0.5ex]{2pt}{.4pt}}}\ULon}

\newcommand{\LC}{\left(}
\newcommand{\RC}{\right)}

\newcommand{\p}{\partial}
\newcommand{\rd}{\mathrm{d}}
\newcommand{\diam}{\text{diam}}

\numberwithin{equation}{section}
\newtheorem{theorem}{Theorem}[section]
\newtheorem{corollary}[theorem]{Corollary}

\newtheorem{lemma}[theorem]{Lemma}

\newtheorem{remark}{Remark}[section]
 
\newcommand{\R}{\mathbb R}

\author[Klingenberg]{Christian Klingenberg}
\address{Department of Mathematics, W\"{u}rzburg University, W\"{u}rzburg 97074, Germany}
\curraddr{}
\email{klingen@mathematik.uni-wuerzburg.de}

\author[Lai]{Ru-Yu Lai}
\address{School of Mathematics, University of Minnesota, Minneapolis, MN 55455, USA}
\curraddr{}
\email{rylai@umn.edu}

\author[Li]{Qin Li}
\address{Department of Mathematics, University of Wisconsin-Madison, Madison, WI 53705, USA}
\curraddr{}
\email{qinli@math.wisc.edu}

\thanks{ 
	}

\title[Nonlinear RTE]{Reconstruction of the emission coefficient in the nonlinear radiative transfer equation}

\date{}

\begin{document}
\maketitle

\begin{abstract}
In this paper, we investigate an inverse problem for the radiative transfer equation that is coupled with a heat equation in a nonscattering medium in $\R^n,\ n\geq 2$. The two equations are coupled through a nonlinear blackbody emission term that is proportional to the fourth power of the temperature. 
By measuring the radiation intensity on the surface of the blackbody, we prove that the emission property of the system can be uniquely reconstructed. In particular, we design a reconstruction procedure that uses merely one set of experiment setup to fully recover the emission parameter.
\end{abstract}

\section{Introduction}
\subsection{Motivation}
Radiative transfer is the physical phenomenon of energy transfer in the form of electromagnetic radiation. The classical model equation for such phenomena is termed the radiative transfer equation (RTE), that encodes absorption, emission, and scattering processes along the radiation. The equation is widely used in optical imaging~\cite{Arridge_1999}, atmospheric science~\cite{petty2006first}, and remote sensing~\cite{review_RS_RTE}.

We study the problem for the radiative transfer equation, when it is combined with a blackbody heat conductance. In particular, the blackbody radiation is coupled with the classical radiative transfer equation through the ``source" term (or the emission term in the RTE). According to the classical theory for the total blackbody emission power, this source term mainly depends on the blackbody temperature, which further satisfies the classical heat equation \cite{minkowycz2000advances,modest2013radiative}. The resulting coupled system reads as follows:
\begin{align}\label{RTE_nonlinear_model}
\begin{cases}
    \partial_t u + \theta\cdot \nabla_x u = -\mu u + \int_{\mathbb{S}^{n-1}} \Phi(\theta',\theta) u(t,x,\theta')\,d{\theta'} + u_b \,, &\\ 
    \partial_t T = \Delta_x T  - u_b + \mu{1\over |\mathbb{S}^{n-1}|}\int_{\mathbb{S}^{n-1}}  u(t,x,\theta)\,d{\theta}\,,& 
\end{cases}
\end{align} 
where $u\equiv u(t,x,\theta)$ describes the radiation intensity at time $t$ on the phase space $(x,\theta)$, with the position $x\in\mathbb{R}^n$ and the direction $\theta\in\mathbb{S}^{n-1}$, and $T$ is the temperature. Here $\mathbb{S}^{n-1}$ is the unit sphere in $\R^n$ and $n\geq 2$.

The first equation of \eqref{RTE_nonlinear_model} is the classical radiative transfer equation with $\mu\equiv \mu(x)$ being the absorption coefficient, and $\int \Phi(\theta',\theta)u(t,x,\theta')\rd{\theta'}$ reflecting the fact that some photons moving in direction $\theta'$ get scattered into $\theta$ direction according to the kernel $\Phi(\theta',\theta)$.
$u_b$ is typically regarded as the source term that introduces new energy into the system. For this particular case, it represents the blackbody emission and is the term that is used to couple $u$ and $T$. The temperature $T$ satisfies the heat equation with a heat sink $u_b$, and a heat source that comes from absorbing the photons. According to the classical theory from total blackbody emission power argument, the blackbody emission is proportional to the fourth power of the temperature:
\[
u_b = \sigma  T^4\,,\\
\]
where $\sigma$, the emission coefficient, is $\sigma_0 \kappa$, with $\sigma_0$ is known as the Stefan-Boltzmann constant, and $\kappa$ characterizes the emissivity of the medium. This parameter is typically unknown for different materials, and it is 	the parameter that we would like to reconstruct by taking measurement on the boundary. For more details, we refer the interested reader to the book \cite{modest2013radiative}, for example, Chapter 10.

\subsection{The setup and main results}\label{sec:setup}
As a start, we confine ourselves to a simpler situation where the scattering is turned off, meaning $\Phi = 0$. We also only look at the steady state solution with $t$ dependence eliminated. The absorption coefficient $\mu$ is assumed to be known, and we target at reconstructing the emission coefficient $\sigma$. 

Let $\Omega$ be an open bounded, connected and strictly convex domain in $\R^n$, $n\geq 2$ with smooth boundary $\partial\Omega$. 
The coupled system now is presented as the following:
\begin{align}\label{nonlinearRTE}
\begin{cases}
    \theta\cdot \nabla_x u + \mu u= \sigma T^4&\hbox{in } \Omega\times\mathbb{S}^{n-1}\,,\\
    \Delta_x T -\sigma  T^4 = -\mu\langle  u\rangle&\hbox{in }\Omega\,,\\
    u =u_B &\hbox{on }\Gamma_-\,,\\
    T =T_B &\hbox{on }\p\Omega\,,
    \end{cases} 
\end{align}
where the notation
\[
\langle  u\rangle(x):={1\over |\mathbb{S}^{n-1}|}\int_{\mathbb{S}^{n-1}} u(x,\theta )\,d\theta 
\]
is the normalized radiation intensity at location $x\in\Omega$.   The sets
\[
\Gamma_\pm := \{(x,\theta)\in \partial\Omega\times\mathbb{S}^{n-1}:\, \pm \theta\cdot n(x)>0\}\,,
\]
collect the boundary coordinates that are pointing out/in of the domain, where the notation $n(x)$ is the unit outer normal vector at $x\in\partial\Omega$, on the boundary. The boundary condition $u_B$ prescribes how many photons are injected at the boundary into the domain.

The forward problem for \eqref{nonlinearRTE} amounts to solving $u$ and $T$ when boundary condition $u_B$ and $T_B$ are given, assuming $\sigma$ and $\mu$ are known. In the inverse problem, it is to reconstruct $\sigma$ from the outgoing measurement $u|_{\Gamma_+}$ for the boundary condition $(u_B,T_B)$.

To make the statement rigorous, we first need to have a mathematical setup in which the coupled system \eqref{nonlinearRTE} makes sense. Throughout the paper we assume $\sigma$ and $\mu$ are compactly supported positive functions in $C^\gamma(\overline\Omega)$, with $0<\gamma<1$. Moreover, there exist positive constants $\sigma_m<\sigma_M<\infty$, and $\mu_m<\mu_M<\infty$ so that the following bounds hold:
\begin{align}\label{sigma}
    0<\sigma_m:=\min_{\overline\Omega} \sigma,\qquad  \|\sigma\|_{C^\gamma(\overline\Omega)}\leq \sigma_M\,,  
\end{align}
and
\begin{align}\label{mu}
    0<\mu_m:=\min_{\overline\Omega} \mu,\qquad \|\mu\|_{C^\gamma(\overline\Omega)}\leq \mu_M\,.
\end{align}

We will show, in Section~\ref{sec:well-posedness}, that under these conditions \eqref{sigma} and \eqref{mu}, the problem \eqref{nonlinearRTE} is well-posed for suitable chosen boundary conditions. More specifically, 
there exists properly defined sets $\mathcal{X}_1\subset C^\gamma(\Gamma_-)$ and $\mathcal{X}_2\subset C^{2,\gamma}(\p\Omega)$ so that when $(u_B,T_B)\in \mathcal{X}=\mathcal{X}_1\otimes\mathcal{X}_2$, 
the problem \eqref{nonlinearRTE} has a unique and positive solution $(u, T)\in  C^\gamma(\overline\Omega \times \mathbb{S}^{n-1}) \times C^{2,\gamma}(\overline\Omega)$. The unique existence of the solution is stated in Theorem~\ref{thm:well_posedness} and its positivity property is presented in Theorem~\ref{thm:positivity}. In particular, the well-posedness of the problem allows us to define the boundary map from $(u_B,T_B)\in  C^\gamma(\Gamma_-)\times C^{2, \gamma}(\p\Omega)$ to $u|_{\Gamma_+}\in C^\gamma(\Gamma_+)$, meaning:
\begin{equation}\label{eqn:def_A}
\begin{aligned}
\mathcal{A}_\sigma:\quad & \mathcal{X}\subset C^\gamma(\Gamma_-)\times C^{2, \gamma}(\p\Omega)&\to\quad&C^\gamma(\Gamma_+)\\
&(u|_{\Gamma_-},T|_{\partial\Omega})&\mapsto\quad &u|_{\Gamma_+}\,.
\end{aligned}
\end{equation}
This places us at the right footing for reconstructing $\sigma$.

We now present the unique identification of $\sigma$ from the map $\mathcal{A}_\sigma$. It is stated in the following theorem.

\begin{theorem}[Uniqueness]\label{main theorem}
Let $\Omega$ be an open bounded, connected and strictly convex domain in $\mathbb{R}^n$, $n\geq 2$ with smooth boundary, and let $\mathcal{A}_{\sigma_j}$ be the boundary operator defined in~\eqref{eqn:def_A} where the system~\eqref{nonlinearRTE} is equipped with the media $\sigma = \sigma_j$ for $j=1,2$. Suppose that $\sigma_j,\, \mu\in C^\gamma(\overline\Omega)$ with $0<\gamma<1$, are compactly supported and satisfy \eqref{sigma}-\eqref{mu}. Given any fixed data $(u_B,T_B)\in \mathcal{X}$, if $\mathcal{A}_{\sigma_1} (u_B, T_B)=\mathcal{A}_{\sigma_2} (u_B, T_B)$, then $\sigma_1=\sigma_2$ pointwisely in $\Omega$.
\end{theorem}

This unique reconstruction result holds true for any dimension with $n\geq 2$. The proof of Theorem~\ref{main theorem} in Section~\ref{sec:uniqueness} relies on the fact that the attenuated $X$-ray transformation $P_\mu f$ can uniquely determine $f$ in $\R^n$ so long as $n\geq 2$. This is to say, suppose the attenuation coefficient $\mu$ is known, then by measuring the integration on a line, the source term $f$ can be reconstructed. We note that the result is rather strong since a single measurement $\mathcal{A}_{\sigma}(u_B,T_B)$ corresponding a fixed data $(u_B,T_B)$ needs to be known. In some sense, only one experiment is sufficient to fully reconstruct the media $\sigma$.
 
Besides the unique reconstruction, we also have a stability estimate.
\begin{theorem}[Stability estimate]\label{thm:stability}
Let $\Omega$ be an open bounded, connected and strictly convex domain in $\mathbb{R}^2$ with smooth boundary, and let $\mathcal{A}_{\sigma_j}$ be the boundary operator defined in~\eqref{eqn:def_A} where the system \eqref{nonlinearRTE} is equipped with the media $\sigma = \sigma_j$ for $j=1,2$. Suppose that $\sigma_j\,,\mu\in C^\gamma(\overline\Omega)$, $0<\gamma<1$, are compactly supported and satisfy \eqref{sigma}-\eqref{mu}. Given any fixed data $(u_B,T_B)\in \mathcal{X}$, then we have the following stability estimate:
\begin{align*}
    \|\sigma_1-\sigma_2\|_{C(\overline\Omega)} 
	&\leq C\|P_\mu^*(\mathcal{A}_{\sigma_1} (u_B,T_B)-\mathcal{A}_{\sigma_2} (u_B,T_B)) \|_{C^\gamma(\overline\Omega)}\,,
\end{align*}	
where the operator $P^*_\mu$ is defined in \eqref{def_Pmu} and the constant $C$ depends on $\Omega,\sigma_M,\mu,\alpha_j,\delta_j$ for $j=1.2$. 
\end{theorem}

As the proof of Theorem~\ref{main theorem}, the proof for obtaining the stability also heavily relies on the two-dimensional reconstruction formula for the attenuated $X$-ray transform in \cite{Novikov}. We can only derive the stability estimate for $n=2$ in Theorem~\ref{thm:stability} due to this dimensional restriction.

The structure of the paper is organized as follows. We will devote Section~\ref{sec:well-posedness} to addressing the well-posedness for the forward problem when boundary conditions are properly chosen. It gives the definition of $\mathcal{A}_\sigma$, as in~\eqref{eqn:def_A} the right footing. We further prove that the temperature $T$ is strictly positive in the whole domain and this positivity turns out to be a crucial factor in the reconstruction procedure. In Section~\ref{sec:uniqueness} we present the recipe, divided into a few substeps, to reconstruct the media $\sigma$ and show the validity of each substep. The stability of the reconstruction is finally shown in Section~\ref{sec:stable}.  	

\section{The forward problem}\label{sec:well-posedness}
In this section, we show that for sufficiently small boundary conditions, the boundary value problem \eqref{nonlinearRTE} is well-posed, and thus $\mathcal{A}_\sigma$ is a well-defined operator. This lays the basic ground for the further exploration of the inverse problem.

Let $\Omega\subset \R^n$, $n\geq 2$ be a bounded open set with smooth boundary, then the H\"older space $C^{k,\gamma}(\overline\Omega),\ 0<\gamma<1$ is the collection of functions so that:
\begin{align*}
\|u\|_{C^{k,\gamma}(\overline\Omega)} := \sum_{|\beta|=k} \sup_{x,y\in\Omega,x\neq y} {|\p^\beta u(x)-\p^\beta u(y)|\over |x-y|^\gamma} +\|u\|_{C^k(\overline\Omega)}<\infty\,.
\end{align*}
Here $k\geq0$ is an integer. When $k=0$, we abbreviate $C^{0,\gamma}(\overline\Omega) = C^{\gamma}(\overline\Omega)$.

Before presenting the theorem about the well-posedness of \eqref{nonlinearRTE}, we first discuss the geometry setup. Since the domain $\Omega$ considered in this paper is connected and strictly convex with $C^\infty$ boundary. For such a domain, there exists a $C^\infty$-function $\xi: \mathbb{R}^{n-1}\to \mathbb{R}$ such that $\Omega$ and its boundary can be described by
\begin{equation}\label{eqn:character}
\Omega =\{x: \xi(x)< 0\} \quad\text{and}\quad\partial\Omega = \{x: \xi(x)=0\}\,.
\end{equation}
Moreover, the strictly convexity here means that there exists a constant $D_0>0$ such that 
\begin{equation}\label{eqn:convex}
  \sum_{ij=1}^n\partial_{ij}\xi(x)a_ia_j \geq D_0|a|^2 
\end{equation}
for all $x$ such that $\xi(x)\leq 0$ and all $a=(a_1,\ldots,a_n)\in \R^n$. This gives $\nabla_x\xi(x)\neq 0$ for any $x\in\partial\Omega$. The construction of $\xi$ uses the distance function $\text{dist}(x, \partial\Omega)$ whose regularity is the same as the regularity of $\partial\Omega$. We refer the reader to Section 14.6 in~\cite{GTPDE}  for more details. The outward normal vector $n(x)$ at $x\in\Omega$ is then given by
\[
n(x)= \frac{\nabla_x\xi(x)}{|\nabla_x\xi(x)|}\,,\quad \forall x\in\partial\Omega\,.
\]

For every $(x\,,\theta)\in\overline\Omega\times\mathbb{S}^{n-1}$, we define the \textit{backward exit time} $\tau_-(x,\theta)\geq 0$ and the \textit{backward exit position} on $\p\Omega$ by:
\begin{equation}\label{eqn:def_tau}
    \tau_-(x,\theta):=\sup\{\{0\}\cup \{\tau>0:\,  x- s\theta\in\Omega\ \hbox{for all } 0<s<\tau\}\}\,,
\end{equation}
and
\begin{equation}\label{eqn:def_x_-}
    x_-(x,\theta) := x-\tau_-(x,\theta)\theta\ \in \p\Omega\,.
\end{equation}

There are some basic properties about the exit time and exit location: 
\begin{lemma}[\cite{Guo2010}, Lemma 2]
Suppose $\Omega\subset\mathbb{R}^n$ is strictly convex and open bounded connected domain and has $C^\infty$ boundary. Suppose that there is a smooth function $\xi$  satisfying ~\eqref{eqn:character} and~\eqref{eqn:convex}. For any $(x, \theta) \in\Omega\times\mathbb{S}^{n-1}$, let $\tau_-$ be the backward exit time defined in~\eqref{eqn:def_tau} and $x_-\in\partial\Omega$ be the exit point defined in~\eqref{eqn:def_x_-}. Then
\begin{itemize}
\item[(1)] $(\tau_-(x,\theta), x_-(x,\theta))$ are uniquely determined for each $(x, \theta) \in\Omega\times\mathbb{S}^{n-1}$;
\item[(2)] If $\theta\cdot n(x_-(x,\theta))\neq 0$, then $\tau_-$ and $x_-$ are smooth functions in $\Omega\times \mathbb{S}^{n-1}$.
\end{itemize}
\end{lemma} 
Noting that for a strictly convex domain $\Omega$, if $(x,\theta)\in \Omega\times \mathbb{S}^{n-1}$, the condition $\theta\cdot n(x_-(x,\theta))\neq0$ always holds true.

We now present our well-posedness result in Theorem~\ref{thm:well_posedness} and the positivity of the solution in Theorem~\ref{thm:positivity}. 

In Theorem~\ref{thm:well_posedness}, we utilize the contraction mapping principle to prove the unique existence of the solution. 
Specifically, there exist sufficiently small parameters $\delta_j$ (with $j=1,2$) such that when $(u_B, T_B)\in \mathcal{X}_0$ defined by
\begin{align}\label{boundary data}
\mathcal{X}_0:=\{(u_B,T_B)\in  C^\gamma(\Gamma_-)\times C^{2,\gamma}(\p\Omega):\, \|u_B\|_{C^\gamma(\Gamma_-)}<\delta_1,\quad \|T_B\|_{C^{2, \gamma}(\p\Omega)} < \delta_2 \},
\end{align} 
the problem \eqref{nonlinearRTE} has a unique solution $(u, T)\in  C^\gamma(\overline\Omega \times \mathbb{S}^{n-1}) \times C^{2,\gamma}(\overline\Omega)$.

In addition to the well-posedness, the positivity property of the solution is essential to the reconstruction of $\sigma$. In Theorem~\ref{thm:positivity}, one requires additional constraints on the lower bounds of the imposed boundary conditions to ensure that the solution $(u,T)$ is bounded away from zero. In particular, there exist sufficient small parameters $\alpha_j$ and $\delta_j$ (with $j=1,2$) such that when $(u_B, T_B)\in \mathcal{X}$, the solution $u$ and $T$ are strictly positive. Here the space $\mathcal{X}$ is defined to be
\begin{align*}
\mathcal{X}:= &\{(u_B,T_B)\in  C^\gamma(\Gamma_-)\times C^{2,\gamma}(\p\Omega): u_B\in\mathcal{X}_1,\quad T_B\in\mathcal{X}_2\} \,,
\end{align*}
where $\mathcal{X}_1\subset C^\gamma(\Gamma_-)$ and $\mathcal{X}_2\subset C^{2,\gamma}(\p\Omega)$ are respectively defined by:
\begin{align*}
\mathcal{X}_1:=\{u_B\in C^\gamma(\Gamma_-):\ 0<\alpha_1 \leq \min_{(x,\theta)\in\Gamma_-}u_B(x,\theta),\quad \|u_B\|_{C^\gamma(\Gamma_-)}<\delta_1 \} \,,
\end{align*}
and
\begin{align*}
\mathcal{X}_2:=\{T_B\in C^{2,\gamma}(\p\Omega):\ 0<\alpha_2 \leq \min_{x\in\p\Omega}T_B(x),\quad\|T_B\|_{C^{2, \gamma}(\p\Omega)} < \delta_2\}\,.
\end{align*}
It poses a slightly higher restriction than $\mathcal{X}_0$ by imposing the lower bound $\alpha_i$.

We first state the well-posedness theorem.
\begin{theorem}[Well-posedness]\label{thm:well_posedness}
Let $\Omega$ be an open bounded, connected and strictly convex domain in $\mathbb{R}^n$, $n\geq 2$ with smooth boundary. 
Suppose that $\sigma_j$ and $\mu$ in $C^\gamma(\overline\Omega)$, $0<\gamma<1$, satisfy \eqref{sigma} and \eqref{mu}. Then the problem~\eqref{nonlinearRTE} is well-posed with small boundary data. In particular, one has that
\begin{itemize}
\item[(1)]  There exist constants $\delta_j>0$, $j=1,2$ sufficiently small, such that for any $(u_B,T_B)\in \mathcal{X}_0$,
$($defined in \eqref{boundary data}$)$, 
the problem \eqref{nonlinearRTE} has a unique solution $(u, T)\in  C^\gamma(\overline\Omega \times \mathbb{S}^{n-1}) \times C^{2,\gamma}(\overline\Omega)$.
\vskip.2cm
\item[(2)] Moreover, there exists a constant $C$ depending on $\Omega, n, \sigma_M, \mu, \delta_j$, $j=1,2$ so that $(u,T)$ satisfies the estimates:
\begin{align}\label{estimate T}
\|T\|_{C^{2,\gamma}(\overline	\Omega)}\leq C\|T_B\|_{C^{2,\gamma}(\p\Omega)}+ C\|u_B\|_{C^\gamma(\Gamma_-)}\,,
\end{align}
and 
\begin{align}\label{estimate u}
\|u\|_{C^\gamma(\overline\Omega\times\mathbb{S}^{n-1})}\leq C\|u_B\|_{C^\gamma(\Gamma_-)} + C \|T\|_{C^{2,\gamma}(\overline\Omega)}^4\,.
\end{align}
\end{itemize}
\end{theorem}
\begin{proof}
	\textbf{Step 1: Perform linearization.}
	Let $(u_0,T_0)$ solve the following problem:
	\begin{align}\label{BVP:RTE linear}
	\begin{cases}
	\theta\cdot \nabla_x u_0 +\mu u_0= 0&\hbox{in }\Omega\times\mathbb{S}^{n-1}\,,\\
	\Delta_x T_0= - \mu \langle u_0\rangle &\hbox{in }\Omega\,,\\
	\end{cases}
	\end{align}
	with the same boundary data
	\[
	u_0|_{\Gamma_-}=u_B\quad \hbox{ and }\quad T_0|_{\partial\Omega}=T_B\,.
	\]
	If $(u,T)$ solves~\eqref{nonlinearRTE}, then we take the difference and call
	\[
	\tilde{u}:=u-u_0\quad \hbox{ and }\quad \widetilde{T}:=T-T_0\,.
	\]
	Then it is easy to see that these remainder terms satisfy
	\begin{align}\label{BVP:RTE nonlinear}
	\begin{cases}
	\theta\cdot \nabla_x \tilde{u} + \mu\tilde{u}=\sigma(T_0+\widetilde{T})^4&\hbox{in }\Omega\times\mathbb{S}^{n-1}\,,\\
	\Delta_x \widetilde{T}=\sigma(T_0+\widetilde{T})^4 - \mu \langle \tilde{u}\rangle &\hbox{in }\Omega\,,\\
	\end{cases}
	\end{align}
	with trivial boundary conditions
	\[
	\tilde{u}|_{\Gamma_-}=0\quad \hbox{ and }\quad \widetilde{T}|_{\partial\Omega}=0\,.
	\]
	
	The $(u_0,T_0)$ system \eqref{BVP:RTE linear} is well-posed. Indeed, since $u_0$ satisfies the transport equation in \eqref{BVP:RTE linear} with boundary data $u_B$, it is straightforward to write down the explicit solution:
	\begin{align}\label{estimate u0}
	u_0(x,\theta) = e^{-\int^{\tau_-(x,\theta)}_0 \mu(x-s\theta)ds} u_B(x-\tau_-(x,\theta)\theta)\,.
	\end{align}
	Considering $u_B\in C^\gamma(\Gamma_-)$, $\mu\in C^\gamma(\overline\Omega)$ and $\tau_-$ is smooth, we have 
	$$
	\max\{\|u_0\|_{C^\gamma(\overline\Omega\times\mathbb{S}^{n-1})}\,, \|\langle u_0\rangle\|_{C^\gamma(\overline\Omega)}\}\leq C  \|u_B\|_{C^\gamma(\Gamma_-)} < C\delta_1 \,.
	$$
	The unique existence of $T_0$ to the elliptic equation is also straightforward. Noting that from Theorem~6.8 in~\cite{GTPDE}, we have the following estimate:
	\begin{align}\label{est T0}
	\|T_0\|_{C^{2,\gamma}(\overline\Omega)} 
	&\leq C \|T_B\|_{C^{2,\gamma}(\p\Omega)}+C \|\langle u_0\rangle\|_{C^{\gamma}(\overline\Omega)} \notag\\
	&\leq C \|T_B\|_{C^{2,\gamma}(\p\Omega)}+C \|u_B\|_{C^\gamma(\Gamma_-)} \leq C(\delta_1+\delta_2)\,,
	\end{align} 
	where the constant $C$ is independent of $T_B$ and $u_B$.
	
	These boil the well-posedness theory for $(u,T)$ system \eqref{nonlinearRTE} down to showing that of $(\tilde{u}\,,\tilde{T})$ system \eqref{BVP:RTE nonlinear}, and is what we will prove below. To that end, we will employ the contracting map argument. 
	
	\textbf{Step 2: Design a contraction map.}
    We consider the following problem:
	\begin{equation}\label{RTE:linear}
	\Delta_x \phi= g  \quad\hbox{in }\Omega\,,\quad \phi=0 \quad\hbox{on }\p\Omega\,.
	\end{equation} 
	According to~\cite{GTPDE}, for any $g\in C^\gamma(\overline\Omega)$, we can define the solution operator of \eqref{RTE:linear} by $\mathcal{L}^{-1}$. Then $\mathcal{L}^{-1}(g)$ is the unique solution of \eqref{RTE:linear}. Moreover, one has that
	\begin{align}\label{elliptic estimate}
     \|\mathcal{L}^{-1}(g)\|_{C^{2,\gamma}(\overline\Omega)}\leq C \|g\|_{C^\gamma(\overline\Omega)}\,,
	\end{align}
	for a constant $C>0$, independent of $g$.  
	
	We define the subset $\mathcal{S}$ in $C^{2,\gamma}(\overline\Omega)$ by
	\begin{align}\label{ellipticReg}
	\mathcal{S} =\{ \varphi\in C^{2,\gamma}(\overline\Omega):\ \varphi|_{\p\Omega}=0,\  \|\varphi\|_{C^{2,\gamma}(\overline\Omega)} < \varepsilon\}\,,
	\end{align}
	where the constant $\varepsilon>0$ will be determined later.
	We are now ready to define the operator $F$ on $\mathcal{S}$ by 
	$$
	F(\varphi)(x):= \mathcal{L}^{-1} (\sigma(T_0+\varphi)^4 -\mu \langle \tilde{u}_\varphi\rangle)\,,
	$$
	for any $\varphi\in\mathcal{S}$, where $\tilde{u}_\varphi$ is the solution to the boundary value problem for the transport equation
	\begin{align}\label{proof reminder u}
	\theta\cdot \nabla_x \tilde{u}_\varphi+ \mu\tilde{u}_\varphi=\sigma(T_0+\varphi)^4 \qquad \hbox{with } \tilde{u}_\varphi|_{\Gamma_-}=0\,.
	\end{align}
	In fact, for any $\varphi\in\mathcal{S}$, one has $\sigma(T_0+\varphi)^4 -\mu \langle \tilde{u}_\varphi\rangle \in C^\gamma(\overline\Omega)$. Then based on the definition of $\mathcal{L}^{-1}$, $F(\varphi)$ is the solution of $\Delta_x F(\varphi)=\sigma(T_0+\varphi)^4 - \mu \langle \tilde{u}_\varphi\rangle$ with trivial boundary condition.
	As a result, solving \eqref{BVP:RTE nonlinear} is reduced to finding a fixed point of $F$ in $\mathcal{S}$.

	Toward this goal, we will show that $F$ is a contracting map which maps from the set $\mathcal{S}$ into itself. Specifically, we will first show for all $\varphi\in \mathcal{S}$, $F(\varphi)\in\mathcal{S}$, and then prove that there is a constant $C_F<1$ so that for all $\varphi_1\,,\varphi_2\in\mathcal{S}$:
	\begin{equation}\label{eqn:contracting}
	\|F(\varphi_1)-F(\varphi_2)\|_{C^{2,\gamma}(\overline\Omega)}\leq C_F\|\varphi_1 - \varphi_2\|_{C^{2,\gamma}(\overline\Omega)}\,.
	\end{equation}
	
	To show $F(\mathcal{S})\subset\mathcal{S}$, we first explicitly write down the solution to the equation~\eqref{proof reminder u}:
	\begin{align}\label{expression tilde u}
	\tilde{u}_\varphi(x,\theta) = \int^{\tau_-(x,\theta)}_0 e^{-\int^s_0 \mu(x-\eta \theta)d\eta} \sigma (T_0+\varphi)^4 (x-s\theta)\,ds,\qquad  (x,\theta)\in\overline\Omega\times\mathbb{S}^{n-1}\,.
	\end{align}
	This leads to the following estimates
	\begin{align}\label{estimate tilde u}
	\max\{\|\langle \tilde{u}_\varphi\rangle\|_{C^\gamma(\overline\Omega)}\,, \|\tilde{u}_\varphi\|_{C^\gamma(\overline\Omega \times \mathbb{S}^{n-1})}\}
	&\leq C \|\sigma (T_0+\varphi)^4\|_{C^{\gamma }(\overline\Omega)}  \notag \\
	&\leq C\sigma_M  \|T_0 + \varphi\|^4_{C^{2,\gamma }(\overline\Omega)}  \notag \\
	&\leq C\sigma_M (\|T_0\|_{C^{2,\gamma}(\overline\Omega)} + \|\varphi\|_{C^{2,\gamma }(\overline\Omega)})^4\,,
	\end{align}
	where we used the fact that $\mu\in C^\gamma(\overline\Omega)$, $T_0,\, \varphi\in C^{2,\gamma}(\overline\Omega)$, $\tau_-$ is smooth, and $\|\sigma\|_{C^\gamma(\overline\Omega)}\leq \sigma_M$ as in \eqref{sigma}. Hereafter, we denote by $C$ any positive constant which may vary from line to line.
	
	Combining \eqref{est T0}, \eqref{elliptic estimate}, \eqref{ellipticReg}, and \eqref{estimate tilde u}, we derive
	\begin{align*}
	\|F(\varphi)\|_{C^{2,\gamma}(\overline\Omega)}
	&=\|\mathcal{L}^{-1} (\sigma(T_0+\varphi)^4 - \mu \langle \tilde{u}_\varphi\rangle)\|_{C^{2,\gamma}(\overline\Omega)}\\
	&\leq C \| \sigma (T_0+\varphi)^4 -\mu \langle \tilde{u}_\varphi\rangle\|_{C^\gamma(\overline\Omega)}\\
	&\leq C (\|T_0\|_{C^{2,\gamma}(\overline\Omega)} +  \|\varphi\|_{C^{2,\gamma}(\overline\Omega)})^4 \\
	&\leq  C (\delta_1+\delta_2 + \varepsilon)^4  \,.
	\end{align*}
	With $\delta_i$ and $\varepsilon$ small enough, one has
	\begin{equation*}
	C(\delta_1+\delta_2 + \varepsilon)^4 < \varepsilon\,,
	\end{equation*}
	meaning $F(\varphi)\in\mathcal{S}$. To show~\eqref{eqn:contracting}, we note
	\begin{align*}
	\|F(\varphi_1) - F(\varphi_2)\|_{C^{2,\gamma}(\overline\Omega)}
	&\leq C\|(\sigma (T_0+\varphi_1)^4 - \mu \langle \tilde{u}_{\varphi_1}\rangle)-(\sigma (T_0+\varphi_2)^4 - \mu\langle \tilde{u}_{\varphi_2}\rangle) \|_{C^\gamma(\overline\Omega)}\\
	&\leq C \sigma_M (\varepsilon + \delta_1+ \delta_2)^3\|\varphi_1-\varphi_2\|_{C^{2,\gamma}(\overline\Omega)} \,,
	\end{align*}
	where we used \eqref{est T0}, \eqref{elliptic estimate}, \eqref{ellipticReg} again, and the estimate
	$$
	\|\mu \langle \tilde{u}_{\varphi_1}\rangle-\mu \langle \tilde{u}_{\varphi_2}\rangle\|_{C^{\gamma}(\overline\Omega)}  \leq C\|\sigma (T_0+\varphi_1)^4-\sigma (T_0+\varphi_2)^4 \|_{C^{\gamma}(\overline\Omega)}\,
	$$
	due to \eqref{expression tilde u}.
	Thus, we further choose $\varepsilon$, $\delta_1$ and $\delta_2$ sufficiently small so that they also satisfy
	\[
	C\sigma_M(\varepsilon +\delta_1+\delta_2)^3 <1\,,
	\]
	which finally leads to~\eqref{eqn:contracting}.
	
	\textbf{Step 3: Summarize the well-posedness result.} Since $F(\mathcal{S})\subset\mathcal{S}$ and $F$ is a contraction map on $\mathcal{S}$, by the contraction mapping principle, there exists a unique fixed point of $F$ on $\mathcal{S}$. We term this fixed point $\widetilde{T}\in\mathcal{S}$, and term the function $\tilde{u}$ the corresponding $\tilde{u}_{\widetilde{T}}$ as defined in~\eqref{expression tilde u}. This fixed point satisfies
	\[
	F(\widetilde{T}) = \widetilde{T}
	\]
	and thus $(\tilde{u}, \widetilde{T})$ is the solution of \eqref{BVP:RTE nonlinear}. Furthermore, $(u,T)\in C^\gamma(\overline\Omega\times\mathbb{S}^{n-1})\times C^{2,\gamma}(\overline\Omega)$ of the form
	$$
	u = u_0+\tilde{u}\quad \hbox{ and }\quad T=T_0+\widetilde{T}
	$$ 
	is then the solution of \eqref{nonlinearRTE}.
	
	\textbf{Step 4: Prove the boundedness of solutions.}
	To show the boundedness of $u$ and $T$, we recall that $(\tilde{u} = \tilde{u}_{\widetilde{T}}\,,\widetilde{T})$ is the solution of \eqref{BVP:RTE nonlinear} and then we obtain
	\begin{align}\label{estimate tilde T}
	\|\widetilde{T}\|_{C^{2,\gamma}(\overline\Omega)}
	&\leq C\| \sigma (\widetilde{T}+T_0)^4 - \mu \langle\tilde{u}_{\widetilde{T}}\rangle\|_{C^\gamma(\overline\Omega)} \notag\\
	&\leq  C\sigma_M (\|T_0\|_{C^{2,\gamma}(\overline\Omega)} +  \|\widetilde{T}\|_{C^{2,\gamma}(\overline\Omega)})^4 \notag\\
	&\leq  C\sigma_M \|T_0\|_{C^{2,\gamma}(\overline\Omega)}^4 + C\sigma_M \|\widetilde{T}\|_{C^{2,\gamma}(\overline\Omega)}\Big(\|\widetilde{T}\|_{C^{2,\gamma}(\overline\Omega)}^3+  \|T_0\|_{C^{2,\gamma}(\overline\Omega)}^2\|\widetilde{T}\|_{C^{2,\gamma}(\overline\Omega)} \notag\\
	&\hskip3cm + \|T_0\|_{C^{2,\gamma}(\overline\Omega)}^3+ \|T_0\|_{C^{2,\gamma}(\overline\Omega)}\|\widetilde{T}\|_{C^{2,\gamma}(\overline\Omega)}^2 \Big)\notag\\
	&\leq  C \sigma_M(\delta_1+\delta_2 )^3  \|T_0\|_{C^{2,\gamma}(\overline\Omega)} + C\sigma_M\Big(\varepsilon^3 + \varepsilon (\delta_1+\delta_2)^2 \notag\\
	&\hskip3cm+  (\delta_1+\delta_2)^3+ \varepsilon^2( \delta_1+ \delta_2) \Big)\|\widetilde{T}\|_{C^{2,\gamma}(\overline\Omega)} \,.
	\end{align}
	Now we let $\varepsilon,\delta_1,\delta_2$ to be small enough such that the second term is controlled by $\frac{1}{2}\|\widetilde{T}\|_{C^{2,\gamma}(\overline\Omega)}$, which eventually leads to
	\begin{align*}
	\|\widetilde{T}\|_{C^{2,\gamma}(\overline\Omega)} &\leq  C\sigma_M ( \delta_1+ \delta_2 )^3   \|T_0\|_{C^{2,\gamma}(\overline\Omega)} \\
	&\leq  C\|T_B\|_{C^{2,\gamma}(\p\Omega)}+C\|u_B\|_{C^\gamma(\Gamma_-)} \,.
	\end{align*}
	This concludes~\eqref{estimate T}. Moreover, the estimate \eqref{estimate u} follows by combing~\eqref{estimate u0} with \eqref{estimate tilde u}.
	
\end{proof}

Furthermore we have the positivity result.
\begin{theorem}[Positivity]\label{thm:positivity}
Under the same condition as in Theorem~\ref{thm:well_posedness}, let $\alpha_1,\,\alpha_2$ satisfy 
$$0< \alpha_j \leq \delta_j,\qquad \alpha_2^4\leq {\alpha_1 \mu_m e^{-d\mu_M} \over \sigma_M},\qquad j=1,2.$$ Then we have the positivity of the solution, namely:
		\begin{align}\label{u away from zero}
		 0<e^{-d\mu_M}\alpha_1 \leq u\quad\ \ \hbox{in }\Omega\times\mathbb{S}^{n-1}\,, 
		\end{align}
		and 
		\begin{align}\label{T away from zero}
		  0<\alpha_2 \leq T\qquad \hbox{in }\Omega\,,
		\end{align}
		where $d:=$ \textnormal{diam}$(\Omega)$ is the diameter of $\Omega$ and $\mu_M$ is defined in \eqref{mu}.  
\end{theorem}
\begin{proof}
To show \eqref{u away from zero} and \eqref{T away from zero}, that is, the boundedness from below, we utilize the contradiction argument. 
We define a new function $P$ by
\[
P(x):=T(x)-\alpha_2.
\]
Since $\alpha_2$ is a constant, 	$P\in C^{2,\gamma}(\overline\Omega)$ satisfies
\begin{align}\label{BVP:elliptic 2}
\left\{\begin{array}{ll}
\Delta_x P =\sigma(T^4 - \alpha_2^4)+ \sigma \alpha_2^4- \mu \langle u\rangle&\hbox{in }\Omega\,,\\
P = T_B -\alpha_2 &\hbox{on }\partial\Omega\,.
\end{array}\right.
\end{align}
To show $T\geq\alpha_2$, it amounts to showing $P\geq 0$ in the entire domain $\Omega$. By the contradiction argument, suppose that it is not true, then the subdomain:
\[
\Omega^-:=\{x\in \Omega:\ T -\alpha_2<0\}
\]
is not empty. Since $P$ is continuous, one has that $\Omega^-$ is an open set. From \eqref{estimate u0}, we can deduce that 
$$
u(x,\theta)=u_0(x,\theta)+\tilde{u}(x,\theta)\geq u_0(x,\theta)\geq e^{-d\mu_M} u_B(x-\tau_-(x,\theta)\theta)\geq e^{-d\mu_M}\alpha_1\,,
$$
where we used the fact that $\tilde{u}\geq 0$ according to~\eqref{expression tilde u}, and $d=\diam(\Omega)$ is the diameter of $\Omega$. It immediately leads to 
$$
e^{-d\mu_M}\alpha_1\leq  \langle u\rangle (x)\,.
$$
Since we choose the constant $\alpha_2$ to satisfy  
$$
0< \alpha_2 \leq \delta_2 \quad\hbox{ and }\quad \alpha_2^4\leq {\alpha_1 \mu_m e^{-d\mu_M} \over \sigma_M}\,,
$$
we can derive that
$$
\sigma(x) \alpha_2^4  \leq \sigma_M\alpha_2^4 \leq \alpha_1 \mu_m e^{-d\mu_M}\leq \mu(x)\langle u\rangle(x)
$$ 
for every $x\in\Omega$. Combining it with $T^4<\alpha_2^4$ in $\Omega^-$, the right-hand side of the equation \eqref{BVP:elliptic 2} actually satisfies
$$
\sigma(T^4 - \alpha_2^4)+ \sigma \alpha_2^4- \mu \langle u\rangle <0\qquad \hbox{in }\Omega^-\,,
$$
which gives that
$$
\Delta_x P  < 0 \qquad\hbox{in }\Omega^-\,.  
$$ 
When $\p\Omega^-\cap\p\Omega =\emptyset$, the function $P=0$ on $\p\Omega^-$, and when $\p\Omega^-\cap\p\Omega\neq \emptyset$, since $T_B-\alpha_2\geq 0$ on $\p\Omega$, the minimum value of $P$ on $\p\Omega^-$ is nonnegative. Thus, $P|_{\partial\Omega^-}\geq 0$. By Maximum principle \cite{GTPDE}, the solution $P$ satisfies that 
$$
\min_{\overline{\Omega^-} } P=\min_{\p \Omega^-} P \geq 0\,,
$$ 
which implies that
$$
P=T-\alpha_2 \geq 0\qquad \hbox{in }\Omega^-\,.
$$
This contradicts to the definition of $\Omega^-$, suggesting that $\Omega^-$ is indeed empty. Then we obtain $P\geq 0$ and $T\geq \alpha_2$ for all $x\in\Omega$.
\end{proof}

\begin{remark}\label{RK:solution}
Based on Theorem~\ref{thm:well_posedness} and Theorem~\ref{thm:positivity}, we have showed that there exist constants $\alpha_j,\delta_j$, $j=1,2$ sufficiently small such that when $(u_B,T_B)\in \mathcal{X}$, the problem \eqref{nonlinearRTE} has a unique and positive solution $(u,T)\in  C^\gamma(\overline\Omega \times \mathbb{S}^{n-1}) \times C^{2,\gamma}(\overline\Omega)$ satisfying the lower bounds \eqref{u away from zero} and \eqref{T away from zero}. Moreover, the solution also satisfies the estimates \eqref{estimate T} and \eqref{estimate u}.
\end{remark}

\begin{remark}
It is worth emphasizing that the well-posedness result holds in general for any sufficiently small boundary data $(u_B,T_B)$ in $C^\gamma(\Gamma_-)\times C^{2, \gamma}(\p\Omega)$. In this paper we crucially need the strict positivity of $T$ as stated in~\eqref{T away from zero} and thus more constraints, such as the lower bounds of $(u_B,T_B)$, are added in order to derive this. Using the contraction mapping to show the well-posedness is a fairly standard technique and has been applied for nonlinear Boltzmann equations \cite{Klar, LUY2020, Li_2020} and for nonlinear elliptic equations \cite{AZ2017, FO19, Kang2002, KU201909, LaiL2020, LLLS2019nonlinear}.
\end{remark}

An immediate corollary is the following result:
\begin{corollary}
Suppose that $\sigma$ and $\mu$ in $C^\gamma(\overline{\Omega})$ satisfy~\eqref{sigma}-\eqref{mu}. The boundary map $\mathcal{A}_\sigma$, defined in~\eqref{eqn:def_A}, is a well-defined map from $\mathcal{X}\subset C^\gamma(\Gamma_-)\times C^{2, \gamma}(\p\Omega)$ to $C^\gamma(\Gamma_+)$. 
\end{corollary}

\begin{proof}
From \eqref{estimate T} and $\eqref{estimate u}$, one has 
\begin{align*} 
\|u\|_{C^\gamma(\Gamma_+)} \leq \|u\|_{C^\gamma(\overline{\Omega}\times\mathbb{S}^{n-1})} 
 \leq C\|u_B\|_{C^\gamma(\Gamma_-)} + C (\|T_B\|_{C^{2,\gamma}(\p\Omega)}+  \|u_B\|_{C^\gamma(\Gamma_-)})^4<\infty\,.  
\end{align*}
Since
$
u|_{\Gamma_+} =\mathcal{A}_\sigma(u|_{\Gamma_-},T|_{\partial\Omega}),
$
we complete the proof.
\end{proof}


\section{The Inverse Problem}\label{sec:uniqueness}
We are now ready to show that the information of $\mathcal{A}_\sigma$ uniquely determines the coefficient $\sigma$. 
The process is divided into three steps.

For one set of boundary condition $(u_B\,,T_B)\in\mathcal{X}$, the solution $u$ and $T$ over the domain are completely preset. We then ran the following reconstruction procedure for recovering the coefficient $\sigma$:
\begin{itemize}
\item In the first step, we first view $\sigma T^4$ as the source term in the transport equation in \eqref{nonlinearRTE}, and write $u$ as an integral of $T$. For this particular set of $(u_B,T_B)$, with the measured data of $u|_{\Gamma_+}$, we have the attenuated X-ray transform of $\sigma T^4$ in every direction. We then use a reconstruction formula in Theorem~\ref{formula} to reconstruct the term $\sigma T^4$.  
\item For the remainder of the first step, since the term $\sigma T^4$ in the transport equation is now recovered, the well-posedness of the problem for the transport equation leads to the unique determination of the solution $u$. 
\item The second step is to substitute the recovered $\sigma T^4$ and $u$ into the elliptic equation in \eqref{nonlinearRTE}. One can then uniquely identify the solution $T$ due to the well-posedness of the elliptic problem.
\item In the third step, since both $\sigma T^4$ (from the first step) and $T$ (from the second step) are now known, it is straightforward to reconstruct $\sigma$ by recalling the positivity of $T$.
\end{itemize}

We note that the entire reconstruction process only requires one set of $u_B, T_B$, meaning so long as one experiment is performed, the media information is completely recovered.

The first step relies on the attenuated X-ray transform result and is the core of our proof. It is summarized in Lemma~\ref{lemma:1st}. The second step, stated in Lemma~\ref{lemma:2nd}, is mainly on solving the boundary value problem for the elliptic equation. Finally, the last step is presented in the proof of Theorem~\ref{main theorem}.

Let $(u_j,T_j)\in C^\gamma(\overline\Omega\times\mathbb{S}^{n-1})\times C^{2,\gamma}(\overline\Omega)$ be the solution to the problem for two different media $\sigma_j$, $j=1,2$:
\begin{align}\label{nonlinearRTE_adjust}
\left\{\begin{array}{ll}
\theta\cdot \nabla_x u_j +  \mu u_j= \sigma_j T_j^4&\hbox{in }\Omega\times\mathbb{S}^{n-1}\,, 	\\
\Delta_x T_j -\sigma_j  T_j^4 = -\mu\langle u_j\rangle&\hbox{in }\Omega\,,
\end{array}\right. 
\end{align} 
with boundary conditions
\[
u|_{\Gamma_-} =u_B\quad \hbox{ and }\quad  T|_{\partial\Omega} =T_B \,,
\]
where $(u_B\,,T_B)\in\mathcal{X}$.

To motivate the reconstruction process for $\sigma$, we regard $S_j = \sigma_jT^4$ as a source term. Then the transport equation reads:
\begin{equation}\label{eqn:u_j_S}
\theta\cdot \nabla_x u_j +  \mu u_j= S_j\,.
\end{equation}
This source term $S_j$ can then be recovered using the attenuated X-ray transform. More specifically, let the attenuated X-ray transform $P_\mu f$ defined to be:
\begin{equation}\label{eqn:X-ray}
P_\mu f(x,\theta) := \int_{\R} e^{-\int^\infty_0 \mu(x+s\theta+\eta\theta) d\eta} f(x+s\theta)\,ds\,, \qquad  (x,\theta)\in \R^{n}\times \mathbb{S}^{n-1}\,.
\end{equation}
It is an easy exercise to see that this is a line integral and the value for all coordinates in a line is unchanged, meaning that
\[
P_\mu f(x,\theta)=P_\mu f(x+t\theta,\theta)\,,\qquad \hbox{for all } t\in \R\,.
\]

Since $\sigma,\mu$ are compactly supported, $S_j =0$, $\mu=0$ outside $\overline\Omega$. Together with $u_1-u_2=0$ on $\Gamma_-$, we obtain
\[
\left( u_1-u_2\right)(x,\theta) = P_\mu (S_1-S_2)(x,\theta)\,,\quad (x,\theta)\in\Gamma_+\,.	
\]
The condition $u_1-u_2=0$ on $\Gamma_+$ immediately leads to $P_\mu (S_1-S_2)=0$ on $\Gamma_+$, see the proof of Lemma~\ref{lemma:1st} for details. Then the question boils down to whether this implies $S_1-S_2=0$ in the domain. To that end we need to cite the following theorem.  
\begin{theorem}[An inversion formula, Theorem~2.1 in \cite{Novikov}]\label{formula}
Suppose $\mu$ and $f$ are compactly supported and continuous functions in $\R^2$ and $\mu$ is given. Define $P_\mu f(x,\theta)$ as in~\eqref{eqn:X-ray} on $(x,\theta)\in T\mathbb{S}^1$,
where $T\mathbb{S}^{1}$ is a subset of $\R^2\times\mathbb{S}^{1}$:
\[
  T\mathbb{S}^{1} :=\{(x,\theta):\ x\in\R^2,\ \theta\in\mathbb{S}^{1},\ x\cdot\theta=0\}\,.
\]
Then $P_\mu f$ on $T\mathbb{S}^1$ uniquely determines $f$ on $\R^2$ with the following formula
\begin{align}\label{Novikov formula}
f(x) &= -{1\over 4\pi} \LC {\p\over \p x_1} - i{\p\over \p x_2}\RC\int_{\mathbb{S}^1}e^{-\int_0^\infty \mu(x-t\theta)dt} \widetilde{m}(x,\theta) P_\mu f(x, \theta) (\theta_1+i\theta_2)\,d\theta \notag\\
& =:  P_\mu^* (P_\mu f)(x)\,,
\end{align} 
where $x=(x_1,x_2)\in\R^2$, $\theta=(\theta_1,\theta_2)\in\mathbb{S}^1$, and $\widetilde{m}$ is calculated from $\mu$ only. Here the adjoint operator $P_\mu^*$ is defined by
\begin{align}\label{def_Pmu}
P_\mu^* (g)(x) := -{1\over 4\pi} \LC {\p\over \p x_1} - i{\p\over \p x_2}\RC\int_{\mathbb{S}^1}e^{-\int_0^\infty \mu(x-t\theta)dt} \widetilde{m}(x,\theta) g(x, \theta) (\theta_1+i\theta_2)\,d\theta \,.
\end{align} 
\end{theorem}
This theorem states that the source $f$ can be uniquely reconstructed using $P_\mu f$ as long as $\mu$ is given. In particular, if $P_\mu f=0$, then $f = 0$. The related results are also found in the literature, see for instance~\cite{ABK98, AMU18, Finch86, Finch, HMS2018, Novikov2002, ST15, SU2011, Sharafutdinov} and the reference therein. It is important to note that albeit the result here is for $n=2$, the results can be easily extended to treat higher dimensions. For $n\geq 3$, one can always reduce it to the two-dimensional case by restricting the attenuated $X$-ray transform to lines in a family of planes whose union forms the whole space, meaning the uniqueness result holds true for all $n\geq 3$.

In this paper, we do not need the explicit expression of $\widetilde{m}$ for our purpose, and thus we refer the interested reader to Theorem~2.1 in \cite{Novikov} for the detailed expression of $\widetilde{m}$.

This allows us to perform our first step. We summarize the result in the following lemma:
\begin{lemma}\label{lemma:1st}
If $\mathcal{A}_{\sigma_1}(u_B, T_B)=\mathcal{A}_{\sigma_2} (u_B, T_B)$ for $(u_B,T_B)\in \mathcal{X}$, then
\[
    \sigma_1 T^4_{1} = \sigma_2T^4_{2} \qquad \hbox{in }\Omega\,. 
\]
Moreover, one can uniquely determine the solution to the transport equation, namely,
\[
    u_1 = u_2 \qquad \hbox{in }\Omega\times\mathbb{S}^{n-1}\,.
\]
\end{lemma} 
\begin{proof}
For a fixed set of boundary condition $(u_B,T_B)\in \mathcal{X}$, there exists a unique solution $(u_j,T_j)$ to the problem \eqref{nonlinearRTE_adjust} for $j=1,2$. Then $u_1-u_2\in C^\gamma(\overline\Omega\times\mathbb{S}^{n-1})$ is the solution to the following boundary value problem:
\begin{align}\label{equation_for_u}
\left\{\begin{array}{ll}
\theta \cdot \nabla_x (u_1-u_2)+\mu (u_1-u_2) =\sigma_1 T_1^4 - \sigma_2 T_2^4 &\mbox{in $ \Omega \times \mathbb{S}^{n-1} $}\,,\\
 u_1-u_2 =0  &\hbox{on }\Gamma_-\,.
\end{array}\right.
\end{align}

Noting that since $\mathcal{A}_{\sigma_1} (u_B, T_B)=\mathcal{A}_{\sigma_2} (u_B, T_B)$, one has $u_1-u_2|_{\Gamma_{+}}=0$ as well. Thus for any point $(x,\theta)\in\Gamma_{+}$, we obtain 
\begin{align}\label{unique atenuated} 
    0= (u_1-u_2)(x,\theta)= \int^{\tau_{-}(x,\theta)}_{0} e^{-\int^{s}_0\mu(x-\eta \theta) \,d\eta} \left(\sigma_1 T^4_{1}-\sigma_2T^4_{2}\right) (x-s\theta) \,ds \,.
\end{align} 
We extend $\mu$, $\sigma_1$, and $\sigma_2$ to the whole space $\R^n$ by zero and keep the same regularity, then we can extend the integral on the right-hand side of \eqref{unique atenuated} to the whole real line. It exactly says that the attenuated X-ray transformation of $\sigma_1 T^4_{1}-\sigma_2T^4_{2}$ is zero, namely,
\[
0=P_\mu(\sigma_1 T^4_{1}-\sigma_2T^4_{2})(x,\theta).
\]
To extend this equality on $\Gamma_+$ to the whole space $\mathbb{R}^n\times\mathbb{S}^{n-1}$, one can see from the definition of the attenuated X-ray transformation that for any $(x,\theta)\in \R^n\times\mathbb{S}^{n-1}$, if there is a parameter $t$ so that $x+t\theta\in \p\Omega$ and $(x+t\theta,\theta)\in\Gamma_+$, then 
$$
P_\mu (\sigma_1 T^4_{1}-\sigma_2T^4_{2})(x+t\theta,\theta)=P_\mu (\sigma_1 T^4_{1}-\sigma_2T^4_{2})(x,\theta)= 0
$$
due to \eqref{unique atenuated}. On the other hand, if there isn't such a $t$, that means the line does not pass through $\Omega$, and since $\Omega$ is strictly convex and $\sigma_j = 0$ outside $\overline\Omega$, the line integral is also trivial. This finally leads to $P_\mu(\sigma_1 T^4_{1}-\sigma_2T^4_{2}) = 0$ for all $\mathbb{R}^n\times\mathbb{S}^{n-1}$.

Therefore, by applying Theorem~\ref{formula}, we obtain	 
$$
\sigma_1 T^4_{1}=\sigma_2T^4_{2}\qquad \hbox{in }\Omega\,,
$$
which suggests that $u_1-u_2$ satisfies the transport equation $\theta\cdot\nabla_x u+\mu u=0$ with trivial boundary conditions. Finally, by applying the well-posedness of the transport equation, we obtain $u_1=u_2$ in $\Omega\times\mathbb{S}^{n-1}$.
\end{proof}
 
The second step is rather standard.
\begin{lemma}\label{lemma:2nd}
If $\sigma_1 T^4_{1}=\sigma_2T^4_{2} $ in $\Omega$ and $
u_1 = u_2$ in $\Omega\times\mathbb{S}^{n-1}$, 
then $T_1=T_2$ in $\Omega$. 
\end{lemma}
\begin{proof}
Since $u_1=u_2$ and $\sigma_1 T^4_1=\sigma_2 T^4_2$, the solution $T_j$, $j=1,2$ satisfies the boundary value problem
\begin{align}\label{nonlinearElliptic}
\left\{\begin{array}{ll}
\Delta_x T_j = \sigma_1 T^4_1- \mu \langle u_1\rangle&\hbox{in }\Omega\,,\\
T_j =T_B & \hbox{on }\p\Omega\,.\end{array}\right.
\end{align} 
By the well-posedness of the elliptic equation, one concludes $T_1=T_2$.
\end{proof}

Now we are ready to show the final step - the unique identification of the coefficient $\sigma$.
\begin{proof}[Proof of Theorem~\ref{main theorem}]
Given a data $(u_B,T_B)\in \mathcal{X}$, from Theorem~\ref{thm:well_posedness} and Theorem~\ref{thm:positivity} (or Remark~\ref{RK:solution}), the solution $T_j$, $j=1,2$ satisfies 
$0<\alpha_2\leq T_j $, which means that $T_j$ is positive at every point in $\Omega$.
Furthermore, from Lemma~\ref{lemma:1st} and Lemma~\ref{lemma:2nd}, we have derived that 
$$
 \sigma_1 T^4_1=\sigma_2 T^4_2\quad \hbox{ and }\quad T_1=T_2\qquad  \hbox{in }\Omega\,,
$$
then we obtain
$$
    \sigma_1T_1^4 - \sigma_2 T^4_2 = (\sigma_1 -\sigma_2)T_1^4=0\,.
$$
Since $T_1>0$ in $\Omega$, it forces that $\sigma_1$ must be equal to $\sigma_2$ at every point in $\Omega$. This completes the proof of the theorem. 
\end{proof}

\begin{remark}
We emphasize that in the proofs of Lemma~\ref{lemma:1st}, Lemma~\ref{lemma:2nd} and Theorem~\ref{main theorem}, the arguments do not rely on any specifically designed incoming data. The results hold true for arbitrarily selected $(u_B,T_B)\in\mathcal{X}$. This means that one single set of data is sufficient to acquire the uniqueness result.
\end{remark}

\section{Stability estimate}\label{sec:stable}
We further study the stability of the reconstruction in this section. Since we apply the inversion formula \eqref{Novikov formula}, the stability estimate of $\sigma$ will be derived only in the two-dimensional case.

Suppose that $(u,T)$ is the solution to the problem \eqref{nonlinearRTE_adjust} with boundary data $(u_B,T_B)\in \mathcal{X}$. Then $T$ is positive away from zero in $\Omega$. Moreover, $u$ is the solution to
$$
\theta\cdot \nabla_x u +  \mu u= \sigma T^4\qquad \hbox{in }\Omega\times\mathbb{S}^{1}\, 
$$
with boundary $u=u_B$ on $\Gamma_-$. Thus, for any point $(x,\theta)\in \Gamma_+$, the solution $u$ can be explicitly written as:
\begin{align}\label{uexpress}
     u(x,\theta)&=e^{-\int_0^{\tau_-(x,\theta)}\mu(x-s\theta)ds}u_B(x-\tau_-(x,\theta)\theta,\theta)+\int^{\tau_{-}(x,\theta)}_{0} e^{-\int^{s}_0\mu(x-\eta \theta) \,d\eta} \sigma T^4(x-s\theta) \,ds\notag \\
     & = e^{-\int_0^{\tau_-(x,\theta)}\mu(x-s\theta)ds}u_B(x-\tau_-(x,\theta)\theta,\theta)+P_{\mu} (\sigma T^4)(x,\theta)\,,
\end{align}
according to the definition of $P_\mu$.
In particular, one has
\begin{align}\label{uexpress 1}
    P_{\mu} (\sigma T^4)(x,\theta)=u(x,\theta)-e^{-\int_0^{\tau_-(x,\theta)}\mu(x-s\theta)ds}u_B(x_-(x,\theta),\theta)\, 
\end{align}
for $(x,\theta)\in \Gamma_+$. 
In particular, the first term on the right-hand side is $u|_{\Gamma_+}=\mathcal{A}_{\sigma}(u_B,T_B)$ and it is known and, moreover, the second term relies on $u_B$ and thus it is also known. These imply that $P_\mu(\sigma T^4)$ on $\Gamma_+$ is now known. By a similar argument in the proof of Lemma~\ref{lemma:1st}, one can recover the full information of $P_\mu(\sigma T^4)$ on $\R^2\times \mathbb{S}^{1}$. By \eqref{Novikov formula} in Theorem~\ref{formula} and \eqref{uexpress 1}, the function $\sigma T^4$ can thus be reconstructed by the inversion formula:
\begin{align}\label{recover_sigmaT}
    \sigma T^4(x) &=P_\mu^*(P_{\mu} (\sigma T^4))(x) \notag\\
    &= P_\mu^* \LC \mathcal{A}_{\sigma}(u_B,T_B) -e^{-\int_0^{\tau_-(x,\theta)}\mu(x-s\theta)ds}u_B(x_-(x,\theta),\theta)\RC(x) \qquad  \hbox{for }x\in\Omega\,.
\end{align}
Noting that the positive solution $T$ can also be reconstructed from the knowledge of $\mathcal{A}_{\sigma}(u|_{\Gamma_-},T|_{\p\Omega})$ as seen in the proof of Theorem~\ref{thm:well_posedness}.  
The coefficient $\sigma$ can then be expressed as:
\begin{align}\label{recover_sigma}
    \sigma(x) =  T^{-4} P_\mu^*\Big( \mathcal{A}_{\sigma}(u_B,T_B) -e^{-\int_0^{\tau_-(x,\theta)} \mu(x-s\theta)ds}u_B(x_-(x,\theta),\theta)\Big)(x)\qquad \hbox{for }x\in\Omega\,,  
\end{align}
due to the linearity of $P^*_\mu$.

We are now ready to show Theorem~\ref{thm:stability}.
\begin{proof}[Proof of Theorem~\ref{thm:stability}]
From \eqref{recover_sigma}, one can derive that
\begin{align}\label{thm:stable1}
	 \|\sigma_1-\sigma_2\|_{C(\overline\Omega)}  
	 &\leq C \|T^{-4}_1 - T_2^{-4}\|_{C(\overline\Omega)}\|P_\mu^*(P_{\mu} (\sigma_1 T_1^4))\|_{C(\overline\Omega)}\notag\\
	 &\quad  + \|T_2^{-4}\|_{C(\overline\Omega)}\|P_\mu^*\LC\mathcal{A}_{\sigma_1}(u_B,T_B) -\mathcal{A}_{\sigma_2}(u_B,T_B) \RC \|_{C(\overline \Omega)}\,.
\end{align}
For the first term on the right-hand side of \eqref{thm:stable1}, by applying the boundedness of $T_j$, that is, $\|T_j\|_{C^{2,\gamma}(\overline\Omega)} \leq C(\delta_1+\delta_2)$, in Theorem~\ref{thm:well_posedness}, one has 
\begin{align}\label{T4}
	\|T^{-4}_1 - T_2^{-4}\|_{C(\overline\Omega)}
	&\leq \|T_1-T_2\|_{C(\overline\Omega)}\|(T_1^3+T_2^3+T_1T_2^2+T_1^2T_2)T^{-4}_1T^{-4}_2\|_{C(\overline\Omega)}  \notag\\
	&\leq C\alpha_2^{-8}(\delta_1+\delta_2)^3  \|T_1-T_2\|_{C(\overline\Omega)}\,.
\end{align}
		
To derive the stability of $\sigma$, it is sufficient to estimate $T_1 - T_2$. Since $T_j$, $j=1,2$, is solution to 
$$
\Delta_x T_j -\sigma_j T_j^4 = -\mu\langle u_j\rangle\,,
$$
with the same boundary $T_B$, the solution $T_j$ satisfies the estimate
\begin{align}\label{T12}
\|T_1-T_2\|_{C^{2,\gamma}(\overline\Omega)} &\leq C\|\sigma_1T^4_1-\sigma_2T^4_2\|_{C^\gamma(\overline\Omega)} + C \|\langle u_1\rangle-\langle u_2\rangle\|_{C^\gamma(\overline\Omega)} \notag\\
&\leq C \|P_\mu^*\LC\mathcal{A}_{\sigma_1}(u_B,T_B) -\mathcal{A}_{\sigma_2}(u_B,T_B) \RC \|_{C^\gamma(\overline\Omega)},
\end{align}
where we used \eqref{recover_sigmaT} and the linearity of $P_\mu^*$. Here for the last inequality of \eqref{T12}, we also used the fact that $\|\langle u_1\rangle-\langle u_2\rangle\|_{C^\gamma(\overline\Omega)}$ is controlled by $\|\sigma_1T^4_1-\sigma_2T^4_2\|_{C^\gamma(\overline\Omega)}$ due to \eqref{expression tilde u}. We then obtain 
    \begin{align*}
    \|T^{-4}_1 - T_2^{-4}\|_{C(\overline\Omega)}\leq C\alpha_2^{-8}(\delta_1+\delta_2)^3  \|P_\mu^*\LC\mathcal{A}_{\sigma_1}(u_B,T_B) -\mathcal{A}_{\sigma_2}(u_B,T_B) \RC \|_{C^\gamma(\overline\Omega)} 
    \end{align*}
    from \eqref{T4} and \eqref{T12}.
Finally, combining it together with \eqref{thm:stable1}, we have
	\begin{align*} 
	&\|\sigma_1-\sigma_2\|_{C(\overline\Omega)}\notag\\
	&\leq C\|P_\mu^*\LC\mathcal{A}_{\sigma_1}(u_B,T_B) -\mathcal{A}_{\sigma_2}(u_B,T_B) \RC \|_{C(\overline\Omega)}\notag\\
	&\quad+ C \|P_\mu^*\LC\mathcal{A}_{\sigma_1}(u_B,T_B) -\mathcal{A}_{\sigma_2}(u_B,T_B) \RC \|_{C^\gamma(\overline\Omega)} \|P_\mu^* (P_\mu(\sigma_1T^4_1))\|_{C(\overline\Omega)}\,.
	\end{align*}
By recalling that  
\[
\|P_\mu^*( P_\mu(\sigma_1T^4_1))\|_{C(\overline\Omega)}=\|\sigma_1 T_1^4\|_{C(\overline\Omega)}\leq C\sigma_M(\delta_1+\delta_2)^4\,.
\]
 
\end{proof}

\textbf{Acknowledgment.}
C. Klingenberg is supported in part by an ERASMUS$+$ mobility grant, R.-Y. Lai is partially supported by the NSF grant DMS-1714490 and Q. Li is supported in part by the NSF grant DMS-1750488. The authors thank Suman Kumar Sahoo for the discussion at the early stage of the project.

\bibliographystyle{abbrv}
\bibliography{transbib}

\end{document}